\RequirePackage{fix-cm}
\documentclass[smallcondensed]{svjour3hack}     % onecolumn (ditto)
\smartqed  % flush right qed marks, e.g. at end of proof
\usepackage{graphicx}
\usepackage{amssymb}
\usepackage{amsfonts}
\usepackage{amsxtra}
\usepackage{bbm, bm}
\usepackage{subfigure}
\usepackage{colonequals}
\usepackage[linesnumbered]{algorithm2e}
\usepackage{breqn}
\usepackage{placeins}
\usepackage{float}
\floatstyle{plain}
\newfloat{procedure}{thp}{lop}
\floatname{procedure}{Procedure}
\usepackage[font=small,labelfont=bf]{caption}
\usepackage[pdftex]{hyperref}
\newcommand{\Po}{\mathcal{P}}

\newcommand{\R}{\mathbb{R}}

\makeatletter
\newcommand\footnoteref[1]{\protected@xdef\@thefnmark{\ref{#1}}\@footnotemark}
\makeatother

\begin{document}

\title{On branching-point selection for trilinear monomials in spatial branch-and-bound: the hull relaxation\footnote{This work extends and presents parts of the first author's doctoral dissertation  \cite{SpeakmanThesis}, and it
 corrects results first announced in the short abstract   \cite{SpeakmanLee2016}.}%\thanks{Grants or other notes
}

\titlerunning{Branching-point selection for trilinear monomials}        % if too long for running head

\author{  Emily Speakman       \and
        Jon Lee %etc.
}


\institute{E. Speakman \at
              Dept. of Industrial and Operations Engineering,
              University of Michigan,
              Ann Arbor. \\
              \email{eespeakm@umich.edu}           %  \\
           \and
          J. Lee \at
               Dept. of Industrial and Operations Engineering,
              University of Michigan,
              Ann Arbor. \\
              \email{jonxlee@umich.edu}
}

\date{}

\maketitle

\begin{abstract}
In Speakman and Lee (2017), we analytically developed the idea of using volume as a measure for comparing relaxations in the context of spatial branch-and-bound.
Specifically, for trilinear monomials, we analytically compared the three possible  ``double-McCormick relaxations'' with the tight convex-hull relaxation. Here, again using volume as a measure, for the convex-hull relaxation of trilinear monomials,
we establish simple rules for determining the optimal branching variable and optimal branching point. Additionally, we compare our results with
current software practice.

%\keywords{global optimization \and spatial branch-and-bound \and trilinear \and monomial \and branching point \and branching variable}
% \subclass{90C26\and 65K05}
\end{abstract}

\section{Introduction}
\label{intro}

In this article, we consider the spatial branch-and-bound (sBB) family of algorithms (see, for example, \cite{Adjiman98},\cite{Ryoo96},\cite{Smith99}, building on \cite{McCormick76}) which aim to find globally-optimal solutions of factorable mathematical-optimization formulations via a divide-and-conquer approach (building on the branch-and-bound approach for discrete optimization, see \cite{Land63} and \cite{LandDoig60}).  Implementations of these sBB algorithms for factorable formulations work by introducing auxiliary variables in such a way as to decompose every function of the original formulation which we can then view as a labeled directed graph (DAG). Leaves correspond to original model variables, and we assume that
the domain of each such model variable is a finite interval.  We have a library of basic functions, including `linear combination' of an arbitrary number variables, and other simple functions of a small number of variables.  The out-degree of each internal node, labeled by a
library function $f\in\mathcal{F}$ that is \emph{not}  `linear combination'  is typically small (say $d_f\leq 3$, for all $f\in\mathcal{F}$).
We assume that we have methods for convexifying each low-dimensional library function $f$ on an arbitrary box domain in $\mathbb{R}^{d_f}$.
From these DAGs, relaxations are composed and refined (see \cite{Belotti09}, for example).
For a given function $f$, the associated DAG can be constructed in more than one way, and therefore sBB has choices to make in  this regard.  Such choices can have a strong impact on the quality of the convex relaxation obtained from the formulation. Because sBB algorithms obtain bounds from these convex relaxations, these choices can have a significant impact on the performance of the algorithm.

There has been substantial research on how to obtain good-quality convex relaxations of graphs of low-dimensional nonlinear functions on various domains (see, for example, \cite{JMW08}, \cite{SA1990}, \cite{Meyer04a}, \cite{Meyer04b}, \cite{Rikun97}, \cite{MF1995}), and some consideration has been given to constructing DAGs in a favorable way.  In particular, in \cite{SpeakmanLee2015}, we  obtained analytic results regarding the convexifications obtained from different ways of treating trilinear monomials, $f=x_1 x_2 x_3$, on non-negative
box domains $\{{\bm{x}} \in \R^3 ~:~ x_i\in[a_i,b_i],~ i=1,2,3\}$.  We computed both the extreme point and inequality representations of the alternative relaxations (derived from iterating McCormick inequalities) and calculated their $4$-dimensional volumes (in the space of $\{(f,x_i,x_j,x_k)\in\mathbb{R}^4\}$) as a comparison measure.  Using volume as a measure gives a way to analytically compare formulations and corresponds to a uniform distribution of the optimal solution across a relaxation; when concerned with non-linear optimization, such a uniform distribution is quite natural.  This is in contrast to linear optimization, where an optimal solution can always be found at an extreme point, and therefore, the distribution of the optimal solution (or optimal solutions) is clearly not uniform across the feasible region.
Experimental corroboration for using volume as a measure of the quality of
relaxations for trilinear monomials appears in
\cite{SpeakmanYuLee2016} (also see \cite{CafieriLeeLiberti10}, concerning
quadrilinear monomials).

Along with utilizing good convex relaxations, other important issues in the effective implementation of sBB for factorable formulations are:
(i) the choice of branching variable, and (ii) the selection of the branching point.  Software developers have tuned their choice of branching point using extensive problem test beds.
%There has been extensive computational research into branching-point selection (e.g., see \cite{Belotti09}).
It is common practice for solvers to branch on the value of the variable at the current solution, adjusted using some method to ensure that the branching point is not too close to either of the interval endpoints.  Often this is done by weighting the interval midpoint and the variable at the optimal solution of the current relaxation, and/or restricting the branching choice to a central part of the interval.  For example, in \cite{Grossmann96} (also see \cite{Epperly:1995:GON:922024}), they suggest branching at the current relaxation point when it is in the middle 60\% of the interval and failing that, branch at the midpoint.
The \verb;ooOPS; software,  see \cite{AgarwalNadeem12},
uses the
solution of an upper-bounding problem as a reference solution, if such a solution is found; otherwise the
solution of the lower-bounding relaxation is used as a reference solution.
\verb;ooOPS; then identifies the non-convex term with
the greatest separation distance with respect to its convex relaxation. The branching  variable is then chosen as the variable whose value
at the reference solution is nearest to the midpoint of its range. But it
is not clear how \verb;ooOPS; then chooses the branching point.

 \cite{SahinidisTawarmalani02} describe a typical way to avoid the interval endpoints by choosing the branching point as
\begin{equation}
\label{branchingpointformula}
 \max\bigg\{a_i+\beta(a_i-b_i),~ \min\Big\{b_i-\beta(b_i-a_i),~ \alpha\hat{x}_i+(1-\alpha)(a_i+b_i)/2\Big\}\bigg\},
\end{equation}
where $\hat{x}_i$ is the value of the branching variable $x_i$ at the current solution. %and $b_i$ (resp., $a_i$) is the current upper (lower) bound of variable $x_i$.
The constants $\alpha \in [0,1]$ and $\beta \in [0, 1/2]$ are algorithm parameters. So, the branching point is the closest point in the interval
\[
[a_i+\beta(a_i-b_i),~ b_i-\beta(b_i-a_i)]
\]
to the weighted combination $\alpha\hat{x}_i+(1-\alpha)(a_i+b_i)/2$ (of the $x$-value in the current optimal solution and the interval midpoint),
thus explicitly ruling out  branching in the bottom and top $\beta$ fraction of the interval.
Note that if  $\beta \leq (1-\alpha)/2$, then there is no such \emph{explicit} restriction,
because already the weighted combination $\alpha\hat{x}_i+(1-\alpha)(a_i+b_i)/2$
precludes branching in the bottom and top $(1-\alpha)/2$ fraction of the interval.

Current available software use a variety of values for the parameters $\alpha$ and $\beta$. The method (mostly) employed by \verb;SCIP; (see \cite{Achterberg2009}, \cite{VigerskeGleixner2016} and the open-source code itself) is to
select the branching point as the closest point in the  middle 60\% of the interval
to the variable value $\hat{x}_i$.  This is equivalent to setting $\alpha=1$ and $\beta=0.2$ and gives an explicit restriction via the choice of $\beta$.    The current \emph{default} settings of \verb;ANTIGONE;\footnote{Private communication with Ruth Misener} (\cite{GloMIQO} and \cite{Misener14}), \verb;BARON;\footnote{Private communication with Nick Sahinidis} (\cite{Sahinidis1996}) and \verb;COUENNE; (see \cite{Belotti09} and the open-source code itself) all have $\beta \leq (1-\alpha)/2$, and so the default branching
point  is simply the weighted combination $\alpha\hat{x}_i+(1-\alpha)(a_i+b_i)/2$; see Table \ref{codeparams}.

\begin{table}[h!]
\begin{center}
  \begin{tabular}{|l|l|l|}
  \hline
  {\bf Solver} & {\bm $\alpha$} & {\bm $\beta$} \\ \hline
  \verb;SCIP; & $ 1.00$ & $  0.20 \quad \not\leq (1-\alpha)/2 = 0.000$  \\ \hline
  \verb;ANTIGONE; & $  0.75$ & $  0.10 \quad \leq (1-\alpha)/2 = 0.125$  \\ \hline
  \verb;BARON; & $  0.70$ & $  0.01 \quad \leq (1-\alpha)/2 = 0.150$  \\ \hline
  \verb;COUENNE; & $  0.25$ & $  0.20 \quad \leq (1-\alpha)/2 = 0.375$  \\ \hline
 \end{tabular}
 \end{center}
 \caption{Default parameter settings}\label{codeparams}
 \end{table}

The different choices are based on combinations of intuition
and substantial empirical evidence gathered by
the software developers.
We note that there is considerable variation in the settings
of these parameters, across the various software
packages. Furthermore, there are other factors (especially in
\verb;BARON;)
 that sometimes supersede selecting a branching point according to formula (\ref{branchingpointformula});
in particular, functional forms involved, the solution of the current relaxation, available incumbent solutions,
complementarity considerations, etc.
Our work is based solely on analyzing a single trilinear monomial,
after
 branching on a variable in that trilinear monomial, with the goal of
helping to guide, and in some cases mathematically
support, the choice of a branching point. Of course variables often appear in multiple functions.  So, when deciding on a branching variable or a branching point, we may obtain conflicting guidance.  But this is an issue with most branching rules, including those developed empirically, and it is always a challenge to find good ways to combine local information to make algorithmic decisions (see \cite{Adjiman98}).  We hope that our results can help influence such decisions. For example, taking weighted averages of scores based on our metric would be a reasonable way to proceed.

\section{Preliminaries}
\label{prelim}

In this work, we focus on trilinear monomials; that is, functions of the form $f=x_1x_2x_3$.  This is an important class of functions for sBB algorithms, because such monomials
may also involve auxiliary variables.  This means that whenever a formulation contains the product of three (or more) expressions (possibly complicated themselves), our results apply.

%In addition, the case of non-zero lower bounds is particularly important; even if many model variables have lower bounds of zero, this will not typically  be the case for an auxiliary variable.  Furthermore, after branching, the lower bound of a variable will always be nonzero for at least one of the two children.

Following \cite{SpeakmanLee2015}, for the variables $x_i \in [a_i,b_i]$, $i=1,2,3$, throughout this paper we assume the following conditions hold:
\begin{equation}
\label{Omega}
\left. \begin{split}
&0 \leq a_i < b_i \;\text{for}\; i=1,2,3, \quad\text{and} \\
&a_1b_2b_3+b_1a_2a_3 ~\leq~ b_1a_2b_3 + a_1b_2a_3 ~\leq~ b_1b_2a_3 + a_1a_2b_3.
\end{split}
\right \}
\tag{$\Omega$}
\end{equation}

To see that the latter two inequalities are without loss of generality, let $\mathcal{O}_i \colonequals a_i(b_jb_k) + b_i(a_ja_k)$, for $i=1,2,3$.  Then we can construct a labeling such that $\mathcal{O}_1 \leq \mathcal{O}_2 \leq \mathcal{O}_3$.  Note that because we are only considering non-negative bounds, the latter part of this condition is equivalent to:

\begin{equation}\frac{a_1}{b_1} \leq \frac{a_2}{b_2} \leq \frac{a_3}{b_3}. \end{equation} This follows from Lemma \ref{lem91} (in the Appendix), and that $b_i>0, \;\;i=1,2,3$.  Also, \emph{it is very important to note that once we have labeled our variables to satisfy \ref{Omega}, our trilinear monomial cannot be treated as  symmetric across variables.}
This condition also arises in the complete characterization of the inequality description for the (polyhedral) convex hull of the graph of the trilinear monomial $f:=x_1x_2x_3$ (in $\R^4$) (see \cite{Meyer04a} and \cite{Meyer04b}).

We introduce the following notation for the convex hull of the graph of
$f:= x_1x_2x_3$ on a box domain:
\[
\Po_{h} \colonequals \text{conv}\left (\left\{(f,x_1,x_2,x_3) \in \R^4 : f= x_1x_2x_3, \;\; x_i \in [a_i,b_i], \; i=1,2,3 \right\}\right).
\]

Instead of referring to convex lower envelopes and concave upper envelopes, we take the view that any given monomial is likely to be composed in many different ways in a complicated formulation, and so we are agnostic about focusing on only one of convex lower envelopes and concave upper envelopes, and rather we look at the convex hull of the graph of the function on the domain of interest (and it's total volume; not just the volume below or above the graph).

The extreme points of $\Po_{h}$ are the eight points that correspond to the $2^3=8$ choices of each $x$-variable at its upper or lower bound (see \cite{Rikun97}). We label these eight points (all of the form $[x_1x_2x_3, x_1, x_2, x_3]^T$) as follows:

{\small\begin{align*}
&v^1:= \left [\begin{array} {c}
b_1a_2a_3\\
b_1\\
a_2\\
a_3
  \end{array} \right] ,\; v^2:= \left [\begin{array} {c}
a_1a_2a_3\\
a_1\\
a_2\\
a_3
  \end{array} \right], \;v^3:= \left [\begin{array} {c}
a_1a_2b_3 \\
a_1\\
a_2\\
b_3
  \end{array} \right], \;v^4:= \left [\begin{array} {cc}
a_1b_2a_3 \\
a_1\\
b_2\\
a_3
  \end{array} \right], \\[5pt] &v^5:= \left [\begin{array} {cc}
a_1b_2b_3 \\
a_1\\
b_2\\
b_3
  \end{array} \right], \;v^6:= \left [\begin{array} {cc}
b_1b_2b_3\\
b_1\\
b_2\\
b_3
  \end{array} \right], \;v^7:=  \left [\begin{array} {cc}
b_1b_2a_3 \\
b_1\\
b_2\\
a_3
  \end{array} \right], \;v^8:= \left [\begin{array} {cc}
b_1a_2b_3 \\
b_1\\
a_2\\
b_3
  \end{array} \right].
  \end{align*}}

The (complicated) inequality description of the convex hull (see \cite{Meyer04a} and \cite{Meyer04b}) is directly used by some global-optimization software (e.g., \verb;BARON; and \verb;ANTIGONE;).  However, other software packages (e.g., \verb;COUENNE; and \verb;SCIP;)  instead use McCormick inequalities iteratively to obtain a (simpler) convex relaxation for  trilinear monomials.  These alternative approaches reflect the tradeoff between using a more complicated but stronger convexification and a simpler but weaker one, especially in the context of global optimization (see \cite{Lee_2007}, for example).

From \cite{SpeakmanLee2015}, we have a formula for the volume of the convex-hull relaxation (additionally, for the various double-McCormick relaxations), parameterized in terms of the upper and lower variable bounds.

\begin{theorem}[see \cite{SpeakmanLee2015}]
\label{TheoremPH} Under \ref{Omega}, we have
\begin{dmath*}
\emph{vol}(\Po_h) = (b_1-a_1)(b_2-a_2)(b_3-a_3)\times \\ \left(b_1(5b_2b_3-a_2b_3-b_2a_3-3a_2a_3) + a_1(5a_2a_3-b_2a_3-a_2b_3-3b_2b_3)\right)/24.
\end{dmath*}
\end{theorem}
Note that due to the asymmetry introduced by \ref{Omega}, the formula does not treat all variables in the same manner. In particular, the role of $x_1$ is quite different than the
roles of $x_2$ and $x_3$ (which can be interchanged). This observation is very important in our analysis that follows.

In the context of branching within sBB,
let $c_i \in [a_i,b_i]$ be the branching point of variable $x_i$.
We obtain two children. By substituting $a_i=c_i$ and $b_i=c_i$ (respectively) for a given variable $x_i$ into the \emph{appropriate formula} (i.e., Theorem \ref{TheoremPH} \emph{after a possible relabeling of the variables}), and summing the results, we obtain the total resulting volume of the relaxations at the two child nodes, given that we branch on variable $x_i$ at point $c_i$. It is important to realize that the volume formula only holds when the labeling \ref{Omega} is respected.  Given when we branch, the bounds in our problem change, we must be careful to ensure that we always use the formula correctly (i.e. if necessary, we relabel the variables to ensure that \ref{Omega} holds).
%It is also important to notice that the role of
%$x_1$ in Theorem \ref{TheoremPH} is special.
%That is, interchanging $x_2$ and $x_3$
%does not affect the computation in general, but this
%is \emph{not} the case with any other relabeling.

% So, in applying Theorem
%\ref{TheoremPH} to the children after branching, we must
%carefully relabel according to \ref{Omega}.

 In \S\ref{x1}, we present our results analyzing optimal branching-point selection for $x_1$. Then, in \S\ref{x2x3}, we present the analysis for  $x_2$ and $x_3$.  Due to the special role of $x_1$, we will see that the analysis (and even the result) is significantly simpler
  for  $x_2$ and $x_3$ than for $x_1$.  In \S\ref{compare},
we analyze branching-variable selection, and in particular, we demonstrate that it
is always best to branch on $x_1$. In  \S\ref{conclusion}, we
 make some concluding remarks. In the Appendix, we
  provide proofs of various technical results that we utilize.

\section{Branching on $x_1$}
\label{x1}

First, we define the following quantities (note that because we assume $b_i>a_i,\; i=1,2,3$, the denominators will not be zero for any valid parameter choice):

\begin{dmath}
\label{q1}
q_1 \colonequals \frac{3a_1a_2a_3+a_1a_2b_3-a_1b_2a_3-3a_1b_2b_3+4b_1a_2a_3-4b_1b_2b_3}{2(3a_2a_3+a_2b_3-4b_2b_3)}; \end{dmath}
\begin{dmath}
\label{q2}
q_2 \colonequals \frac{a_1+b_1}{2};
\end{dmath}
%and
\begin{dmath}
\label{q3}
q_3 \colonequals \frac{4a_1a_2a_3-4a_1b_2b_3+3b_1a_2a_3+b_1a_2b_3-b_1b_2a_3-3b_1b_2b_3}{2(4a_2a_3-b_2a_3-3b_2b_3)}.
\end{dmath}
Next, we refer to Procedure \ref{fig:flowchart} which depicts a procedure for choosing a branching point when branching on variable $x_1$.
Note that $q_1$ is not used in the procedure, but it is used in the analysis of the procedure.

\begin{procedure}[h!]
\begin{center}
     \includegraphics[width=0.95\textwidth,keepaspectratio]{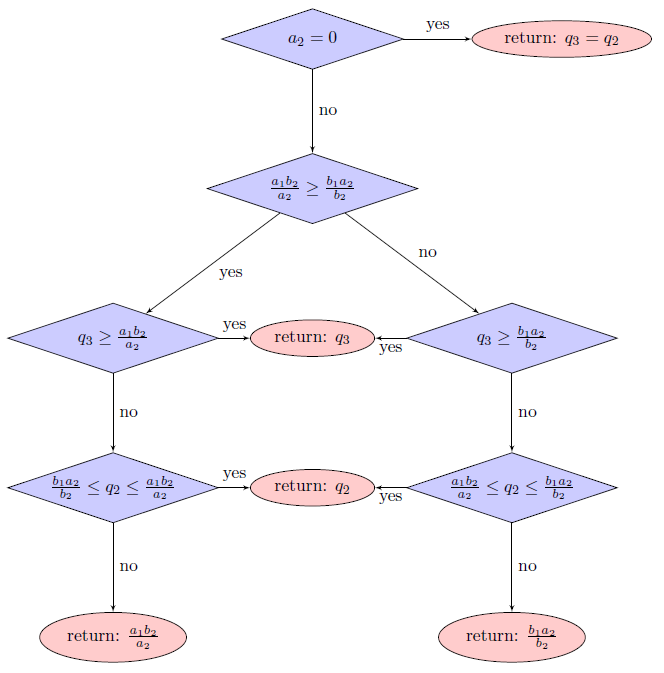}
      \caption{Output is the optimal branching point when branching on variable $x_1$}
    \label{fig:flowchart}
 \end{center}
 \end{procedure}

 First, consider what happens when we pick a branching variable $x_i$, and branch at a given point $c_i$: we obtain two children, now with different bounds on the branching variable.  The upper bound of the branching variable in the \emph{left child} becomes the value of the branching point, as does the lower bound of the branching variable in the \emph{right child}.
That is, the domain of $x_i$ for the left child is $[a_i,c_i]$,
and the domain of $x_i$ for the right child is $[c_i,b_i]$.
 We reconvexify the two children using our chosen method of convexification (i.e., the convex hull), and we can sum the volumes from both children to obtain the total volume when branching at that given point.  We are interested in finding the branching point that leads to the least total volume.  For an example of this principle in a lower dimension, see the diagram of Figure \ref{figsBB} which illustrates reconvexifying after branching in sBB.  Here, because we have a one dimensional function, the graph of the function is a set in $\R^2$.  Therefore, in the context of this diagram, we wish to find the branching point that minimizes the sum of the areas of the two diagonally striped (green) regions.  Clearly this depends on the choice of convexification method.

\begin{figure}[h!t]
\begin{center}
    \includegraphics[width=0.7\textwidth,keepaspectratio]{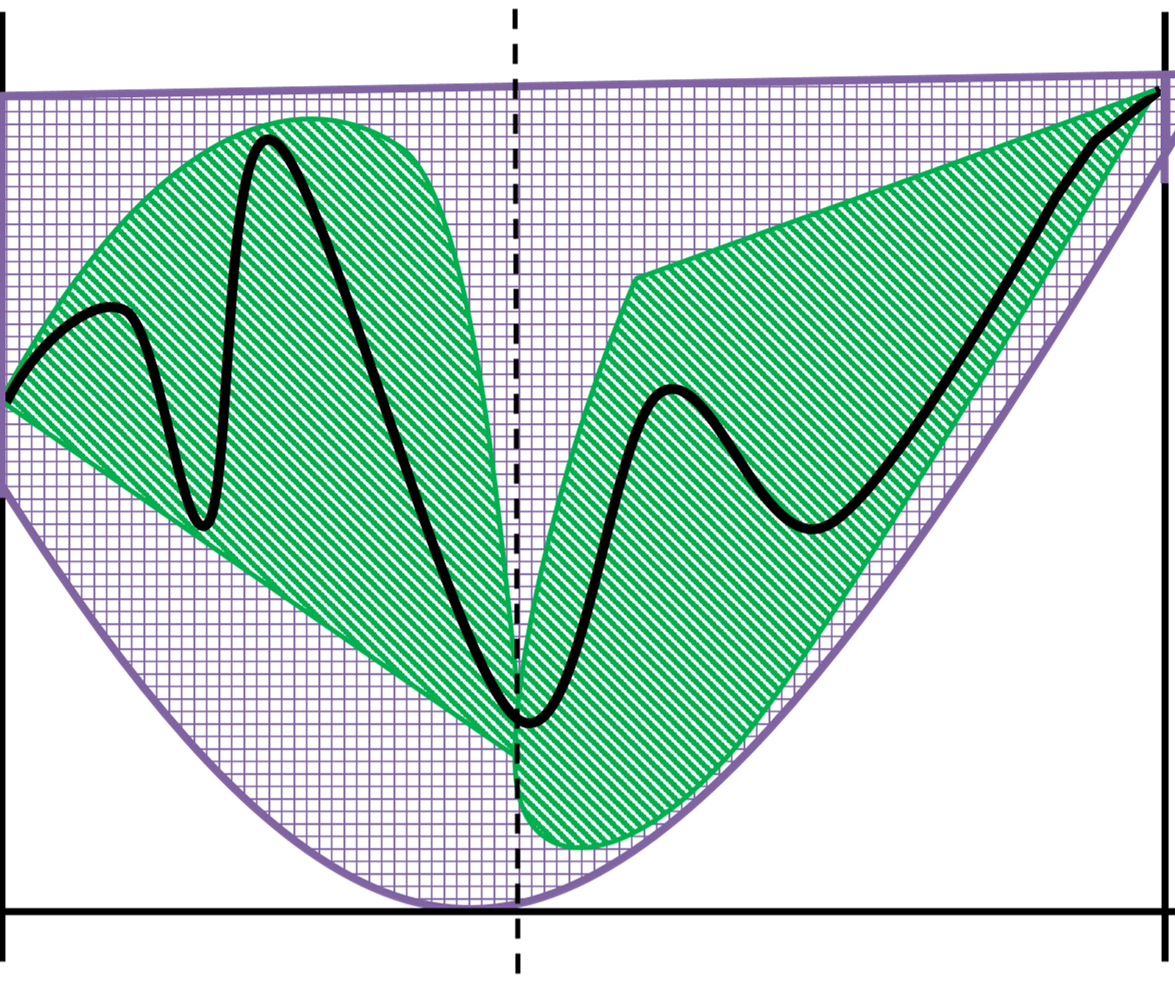}
       \caption{Illustration of sBB for a 1-dimensional function}
    \label{figsBB}
    \end{center}
    \vspace{5mm}
 \end{figure}

We can compute the volume of the relaxation for each of the children using Theorem \ref{TheoremPH} (i.e., Theorem 4.1 from \cite{SpeakmanLee2015}).
To ensure that we compute the appropriate volumes, we need to check that as the bounds on the branching variable change, we still respect the labeling \ref{Omega}.  To illustrate this, consider the left child obtained by branching on variable $x_1$ at some point $c_1 \in [a_1,b_1]$.  For this left child, the lower bound on the branching variable remains the same and the new upper bound is $c_1$.  We can see that if $c_1$ is `close enough' to $b_1$, then \ref{Omega} will  remain satisfied, however as $c_1$ decreases, there comes a point where  the labeling must change.  By simple algebra, we calculate that this critical point is when $\frac{a_1}{c_1} = \frac{a_2}{b_2} \Leftrightarrow c_1 = \frac{a_1b_2}{a_2}$ (assuming for now that $a_2>0$).  We can consider the right child in the same manner.  On the right, the upper bound on the branching variable remains the same, and the new lower bound is $c_1$.  When $c_1$ is close to $a_1$, \ref{Omega} will  remain satisfied; however, as $c_1$ becomes larger, eventually the labeling must change.  This critical point for the right child is at $\frac{c_1}{b_1} = \frac{a_2}{b_2} \Leftrightarrow c_1 = \frac{b_1a_2}{b_2}$.

We note that because of the structure of the volume function of the convex hull, (see Theorem \ref{TheoremPH}), the second and third variables are interchangeable.  This means that we do not need to consider what happens when the bounds vary enough for $x_1$ to be relabeled as $x_3$.

Before stating Theorem \ref{x1flowchart}, we need to clarify some definitions and state four technical lemmas that will be needed in the proof.  The proofs of these four lemmas can be found in the Appendix.

We first define \begin{dmath*}
V(l_1,u_1,l_2,u_2,l_3,u_3) \colonequals (u_1-l_1)(u_2-l_2)(u_3-l_3)\times
\left(u_1(5u_2u_3 - l_2u_3 - u_2l_3 - 3l_2l_3) + l_1(5l_2l_3 - u_2l_3 - l_2u_3 -3u_2u_3)\right)/24,
\end{dmath*}
to be the volume of the convex hull with variable lower bounds $l_i$ and upper bounds, $u_i$, for $i=1,2,3$.

Then, we define the following parameterized function:

\begin{dmath}
\label{TV1}
TV(c_1) \colonequals \begin{cases} V_1(c_1),  & a_1 \leq c_1 \leq \frac{b_1a_2}{b_2}; \\
V_2(c_1),& \frac{b_1a_2}{b_2} < c_1 < \frac{a_1b_2}{a_2}; \\
V_3(c_1),&   \frac{a_1b_2}{a_2} \leq c_1 \leq b_1,           \end{cases}
\end{dmath}

\noindent where:
\begin{align*}V_1(c_1) &\colonequals V(a_2,b_2,a_1,c_1,a_3,b_3) + V(c_1,b_1,a_2,b_2,a_3,b_3), \\
V_2(c_1) &\colonequals V(a_2,b_2,a_1,c_1,a_3,b_3) + V(a_2,b_2,c_1,b_1,a_3,b_3),  \\
V_3(c_1) &\colonequals V(a_1,c_1,a_2,b_2,a_3,b_3) + V(a_2,b_2,c_1,b_1,a_3,b_3).\end{align*}

And finally the second parameterized function:

\begin{dmath}
\label{TV2}
\widehat{TV}(c_1) \colonequals \begin{cases} V_1(c_1)  & a_1 \leq c_1 \leq \frac{b_1a_2}{b_2}; \\
V_4(c_1)& \frac{b_1a_2}{b_2} < c_1 < \frac{a_1b_2}{a_2}; \\
V_3(c_1)&   \frac{a_1b_2}{a_2} \leq c_1 \leq b_1,           \end{cases}
\end{dmath}

\noindent where $V_1(c_1)$ and $V_3(c_1)$ are defined as before and:
$$V_4(c_1) \colonequals V(a_1,c_1,a_2,b_2,a_3,b_3) + V(c_1,b_1,a_2,b_2,a_3,b_3).$$

Both $TV(c_1)$ and $\widehat{TV}(c_1)$ are piecewise-quadratic functions in $c_1$. We can easily observe this by noticing that $V$ is the product of a pair of multilinear functions in the parameters.

\begin{lemma}
\label{HlessB}
Given that the upper- and lower-bound parameters respect the labeling \ref{Omega}, and $\frac{b_1a_2}{b_2} \leq \frac{a_1b_2}{a_2}$,
\begin{equation*}V_1\left(\frac{b_1a_2}{b_2}\right)=V_2\left(\frac{b_1a_2}{b_2}\right) \geq V_2\left(\frac{a_1b_2}{a_2}\right) = V_3\left(\frac{a_1b_2}{a_2}\right).
\end{equation*}
\end{lemma}

\begin{lemma}
\label{HlessB2}
Given that the upper- and lower-bound parameters respect the labeling \ref{Omega}, and $\frac{b_1a_2}{b_2} > \frac{a_1b_2}{a_2}$,
\begin{equation*}
V_1\left(\frac{a_1b_2}{a_2}\right)=V_4\left(\frac{a_1b_2}{a_2}\right) \geq V_4\left(\frac{b_1a_2}{b_2}\right) = V_3\left(\frac{b_1a_2}{b_2}\right).
\end{equation*}
\end{lemma}

  \begin{lemma}
 \label{c1q1}
 Given that the parameters satisfy the conditions \ref{Omega}, and furthermore, $\frac{b_1a_2}{b_2} \leq \frac{a_1b_2}{a_2}$, we have $$q_1 \geq \frac{b_1a_2}{b_2}.$$
 \end{lemma}

\begin{lemma}
  \label{c2q1}
  Given that the parameters satisfy the conditions \ref{Omega}, and furthermore, $\frac{b_1a_2}{b_2} \geq \frac{a_1b_2}{a_2}$, we have $$q_1 \geq \frac{a_1b_2}{a_2}.$$
 \end{lemma}

We are now ready to state the theorem.

\begin{theorem}
\label{x1flowchart}
Assume initial bounds $a_1,b_1,a_2,b_2,a_3,b_3$ that satisfy \ref{Omega} and that we branch on $x_1$.  Furthermore, assume that $q_2$ and $q_3$ are defined as in Equation \ref{q2} and Equation \ref{q3}.  Procedure \ref{fig:flowchart} gives the optimal branching point with respect to minimizing the sum of the volumes of the two convex-hull relaxations of the children.
\end{theorem}

\begin{proof}

Given our earlier discussion, it is natural to think about three cases.  First, when $a_2=0$ (we refer to this as Case 0). Second (Case 1), when
$$a_2 \neq 0 \;\;\;\; \text{and}  \;\;\;\;\frac{b_1a_2}{b_2} \leq \frac{a_1b_2}{a_2}\quad \Longleftrightarrow \quad \frac{a_2^2}{b_2^2} \leq \frac{a_1}{b_1}\quad \Longleftrightarrow \quad b_1a_2^2 \leq a_1b_2^2~, $$
and third (Case 2), when
$$a_2 \neq 0 \;\;\;\; \text{and} \;\;\;\; \frac{b_1a_2}{b_2} > \frac{a_1b_2}{a_2} \quad \Longleftrightarrow \quad \frac{a_2^2}{b_2^2} > \frac{a_1}{b_1} \quad \Longleftrightarrow \quad b_1a_2^2 > a_1b_2^2~ .$$
The case of equality, i.e., $\frac{b_1a_2}{b_2} = \frac{a_1b_2}{a_2}$, is arbitrarily included with Case 1.  In fact, when equality holds, the analysis that follows is simplified, and it could be contained in either of the cases.

Depending on the case, the necessary relabeling to ensure $\Omega$ remains satisfied is different, and the functions we defined as $V_i(c_1), \; i=1\dots4$ reflect these different relabelings.  For an illustration of when the variable labeling must change on the left child, the right child, or on both children to ensure that \ref{Omega} remains satisfied (as the branching point varies), see Figure \ref{fig:branch}.

%Finally, we note that we must consider separately what happens when $a_2=0$ because when this happens, our case analysis involves division by zero.

%\begin{figure}[h!t]
%\begin{center}
 %   \includegraphics[width=0.95\textwidth,keepaspectratio]{pics/Case0}
 %      \caption{Variable labeling as the branching point varies in Case 0}
 %   \label{fig:branch0}
 %   \end{center}
 %   \vspace{5mm}
 %\end{figure}

\begin{figure}[h!t]
\begin{centering}
 \subfigure[Case 0]{\includegraphics[width=0.95\textwidth,keepaspectratio]{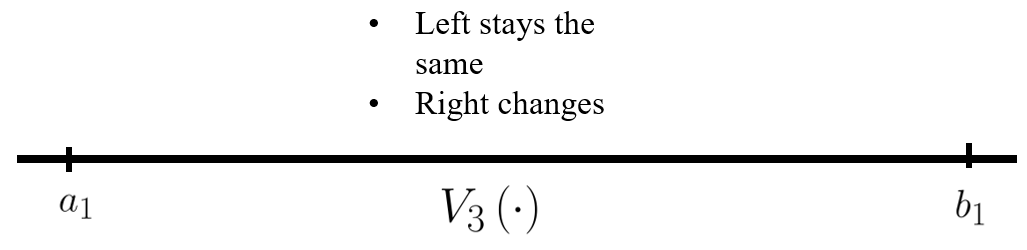}}\\[7mm]
  \subfigure[Case 1]{\includegraphics[width=0.95\textwidth,keepaspectratio]{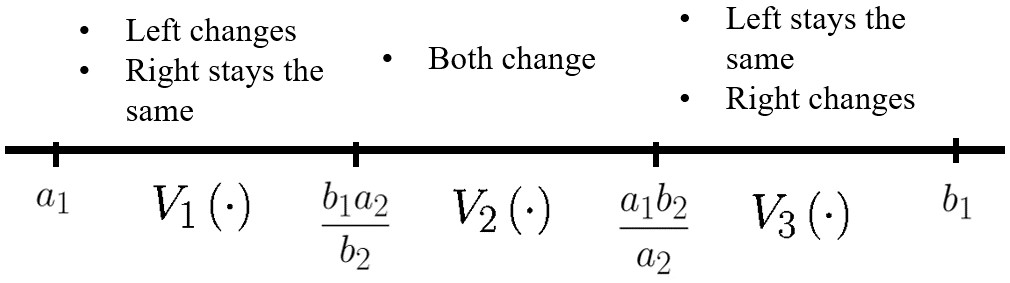}}\\[7mm]
  \subfigure[Case 2]{\includegraphics[width=0.95\textwidth,keepaspectratio]{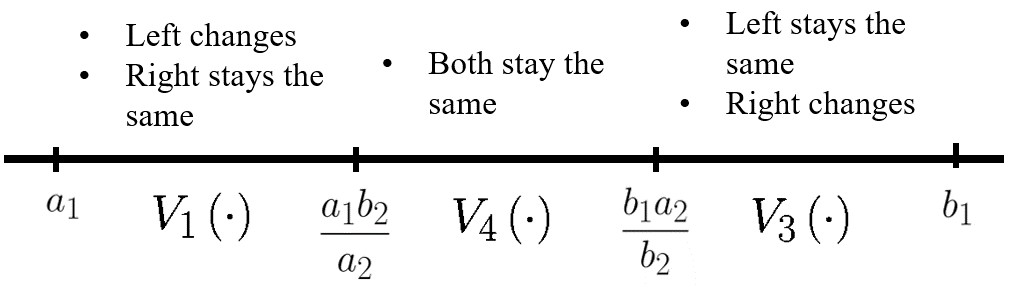}}
    \caption{Variable labeling as the branching point varies for the three cases}
    \label{fig:branch}
    \end{centering}
 \end{figure}

\paragraph{Case 0: $a_2=0$.}
\label{case0}

From the condition \ref{Omega}, we know that $a_2=0 \Rightarrow a_1=0$.  In this special case, the labeling for the left child does not change no matter how small the upper bound becomes.  Conversely, the labeling for the right child changes as soon as the lower bound becomes positive.  We therefore have the picture shown in Figure \ref{fig:branch}, and the function (i.e. $V_3(c_1)$) describing the sum of the volumes of the two child relaxations over the entire domain, $[a_1,b_1]$, does not need to be defined in a piecewise manner.  As we will see shortly, this function is a convex quadratic, and therefore it is easy to check (by calculating where the derivative is zero) that in this special case the minimizer of this function is $q_3$ (defined above) and this is the minimizer of the total volume of the two children.  Furthermore, when $a_2=0$ (and therefore $a_1=0$), this minimizer simplifies to $\frac{b_1}{2}=\frac{a_1+b_1}{2}=q_2$, the midpoint of the interval.

\paragraph{Case 1: $\frac{b_1a_2}{b_2} \leq \frac{a_1b_2}{a_2}$.}
\label{case1}

As illustrated in Figure \ref{fig:branch}, in Case 1, the function describing the sum of the volumes of the child relaxations is $TV(c_1)$ (Equation \ref{TV1}).  It is straightforward to check that the function $TV(c_1)$ is continuous over its domain.  Furthermore, by observing that the leading coefficient of each piece ($V_i(c_1), \; \;i=1,2,3$) is positive for all parameter values satisfying \ref{Omega}, we conclude that each piece is strictly convex.  We are able to claim \emph{strict} convexity because we assume $b_i > a_i$ for all $i$. Using this fact, for each coefficient below we observe that each multiplicand in the numerator is strictly positive and therefore each leading coefficient is strictly positive.

\begin{align*}
&\text{The coefficient of $c_1^2$ in the quadratic function $V_1(c_1)$ is:} & \\
& \hspace*{32.5mm} \frac{(b_3-a_3)(b_2-a_2)(6(b_2b_3-a_2a_3) + 2b_3(b_2-a_2))}{24} > 0. \\
&\text{The coefficient of $c_1^2$ in the quadratic function $V_2(c_1)$ is:}  & \\
& \hspace*{22mm} \frac{(b_3-a_3)(b_2-a_2)(4(b_2b_3-a_2a_3)+2(b_3+a_3)(b_2-a_2))}{24} > 0. \\
&\text{The coefficient of $c_1^2$ in the quadratic function $V_3(c_1)$ is:}& \\
& \hspace*{32.5mm}\frac{(b_3-a_3)(b_2-a_2)(6(b_2b_3-a_2a_3)+2a_3(b_2-a_2))}{24} > 0.
\end{align*}

Figure \ref{fig:piecewisequad} gives some idea of what this function could look like. The example depicts a globally convex function and we are yet to prove that this will always be the case.  However, in later analysis (Theorem \ref{globalconvexity} in \S \ref{sec:globcon}) we will demonstrate that global convexity always holds.

\begin{figure}[h!]
\begin{center}
    \includegraphics[width=0.85\textwidth,keepaspectratio]{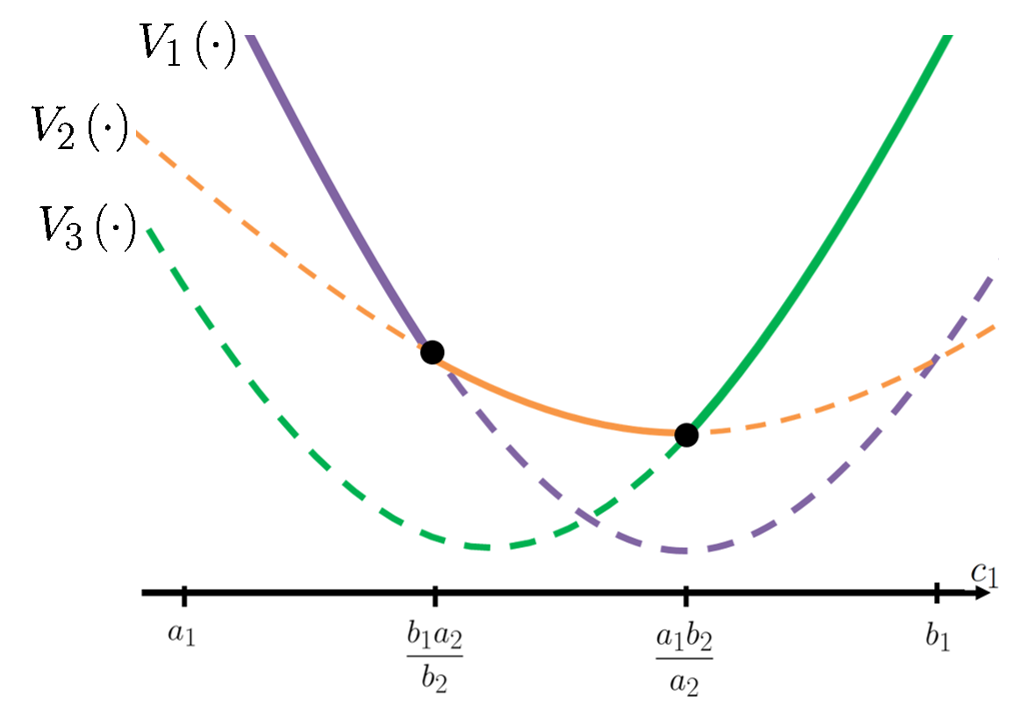}
    \caption{Illustration of a (globally convex) continuous piecewise-quadratic function}
    \label{fig:piecewisequad}
     \end{center}
 \end{figure}

Now that we know that $TV(c_1)$ has this structure, to find the global minimizer over the domain $[a_1,b_1]$, we can simply find the local minimizer on each of the three pieces and pick the point with the least function value.  Because we have convex functions, the local minimum of a given piece will either occur at the global minimizer of $V_i(c_1)$ (if this occurs over the appropriate subdomain), or at one of the end points of the subdomain.  Therefore, to find the local minimizer for a given segment, we first find the global minimizer of $V_i(c_1)$ over the entire real line and check if it occurs in the interval; if so, it is the local minimizer, if not, we examine the interval end points to locate the local minimizer.  We can then compare the function value of the local minimizer of each of the three pieces to find the global minimizer of $TV(c_1)$, i.e., the branching point that obtains the least total volume.

Given that each $V_i$ is a parameterised convex-quadratic function in $c_1$, it is easy to use a computer algebra system to calculate the following:\\
The minimum of $V_1(c_1)$ occurs at:
\begin{equation*}
c_1=\frac{3a_1a_2a_3+a_1a_2b_3-a_1b_2a_3-3a_1b_2b_3+4b_1a_2a_3-4b_1b_2b_3}{2(3a_2a_3+a_2b_3-4b_2b_3)}=q_1.
\end{equation*}
The minimum of $V_2(c_1)$ occurs at:
\begin{equation*}
c_1=\frac{a_1+b_1}{2}=q_2.
\end{equation*}
The minimum of $V_3(c_1)$ occurs at:
\begin{equation*}
c_1=\frac{4a_1a_2a_3-4a_1b_2b_3+3b_1a_2a_3+b_1a_2b_3-b_1b_2a_3-3b_1b_2b_3}{2(4a_2a_3-b_2a_3-3b_2b_3)}=q_3.
\end{equation*}

Therefore, the candidate points for the minimizer are $a_1$, $\frac{b_1a_2}{b_2}$, $\frac{a_1b_2}{a_2}$, $b_1$, $q_1$, $q_2$ and $q_3$.  We can immediately discard $a_1$ and $b_1$ because these are both equivalent to not branching. By branching and reconvexifying over the two children, we can never do worse with regard to volume.  Therefore, we have five points to consider.  For a given set of parameters, it is straightforward to evaluate and check which of these five points is the minimizer.  However, making use of the following observations, we can further reduce the possibilities.

If $q_1$ were to be the global minimizer, then it must fall in the appropriate subdomain; i.e., it must be that $q_1 \leq \frac{b_1a_2}{b_2}$.  However, by Lemma \ref{c1q1}, in Case 1 we always have $q_1 \geq \frac{b_1a_2}{b_2}$. Therefore, we can discard $q_1$ as a candidate point for the minimizer because for it to be the minimizer, this quantity would have to be exactly equal to $\frac{b_1a_2}{b_2}$, which is already on the list of candidate points.

Now, consider the quantities: \begin{equation}\label{q1more} q_1 - q_2 = \frac{(b_3-a_3)(b_1a_2-a_1b_2)}{2(4b_2b_3-a_2b_3-3a_2a_3)}\geq 0,\end{equation} and \begin{equation}\label{q3less}  q_3 - q_2 = \frac{(a_3-b_3)(b_1a_2-a_1b_2)}{2(3b_2b_3+b_2a_3-4a_2a_3)}\leq 0.\end{equation}

The inequalities follow from $b_i > a_i$, $i=1,2,3$, and Lemma \ref{lem91} in the Appendix.  We therefore have: \begin{equation}\label{minordering}q_1 \geq q_2 = \frac{a_1+b_1}{2} \geq q_3.\end{equation}

From this, we can observe that if $q_3 \geq \frac{a_1b_2}{a_2}$, then $q_2 \geq q_3 \geq \frac{a_1b_2}{a_2}$, and therefore $q_3$ is the minimizer.  This is because neither $q_1$ nor $q_2$ fall in their key intervals (i.e. in the appropriate subdomain); furthermore, by the definition of $q_3$ as the minimizer of $V_3$, we must have that $V_3(q_3) \leq V_3\left(\frac{a_1b_2}{a_2}\right)$, and by Lemma \ref{HlessB}, we have that $V_3\left(\frac{a_1b_2}{a_2}\right) \leq V_2\left(\frac{b_1a_2}{b_2}\right)$.

If this does not occur, i.e. $q_3 < \frac{a_1b_2}{a_2}$, then if $\frac{b_1a_2}{b_2} \leq \frac{a_1+b_1}{2} \leq \frac{a_1b_2}{a_2}$, the midpoint $q_2$ is the minimizer.  This is because under these conditions, $q_2$ is the only minimizer that occurs in the `correct' function piece, and by definition of $q_2$ as the minimizer of $V_2$, the function value is not more than at either of the end points.

Otherwise, if none of the above occurs (i.e., none of the intervals contain their function global minimizer), we have that $\frac{a_1b_2}{a_2}$ is the minimizer by Lemma \ref{HlessB}.

%%%%%%%%%%%%%%%%%%%%%%%%%%%%%%%%%%%%%%%%%%%%%%%%%%%%%%%%%%%%%%%%%%%%%%%%%%%%%%%%%%%%%%%%%%%%%%%%%%%%%%%%%%%%%%
\paragraph{Case 2: $\frac{b_1a_2}{b_2} > \frac{a_1b_2}{a_2}$.}
\label{case2}

In this second case, for a given problem with initial upper and lower bounds ($a_1$, $b_1$, $a_2$, $b_2$, $a_3$, $b_3$), the sum of the volumes of the two child relaxations after branching at point $c_1$, is given by the function $\widehat{TV}(c_1)$ (Equation \ref{TV2} and illustrated in Figure \ref{fig:branch}).  This is similar, but distinct, from the function in Case 1.

Recall that this is a piecewise-quadratic function in $c_1$, and, as before, it is simple to check that the function is continuous over its domain.  Furthermore, by observing that the leading coefficient of each piece is positive for all parameter values satisfying \ref{Omega}, we know that each piece is strictly convex.  Strict convexity comes from the knowledge $b_i>a_i, \; i=1,2,3$.

\begin{align*}&\text{The coefficient of $c_1^2$ in the quadratic function $V_4(c_1)$ is:} \\
& \hspace*{64mm} \frac{8(b_3-a_3)(b_2-a_2)(b_2b_3-a_2a_3)}{24} > 0.\end{align*}

Therefore, we can take the same approach as before to find the global minimizer: first find the local minimizer for each segment.  We do this by finding the global minimizer for the appropriate function ($V_i(c_1)$), over the whole real line and checking if it occurs in the segment.  If it does, we have found the minimizer for that segment, if not, we examine the interval end points.  We then compare the minimum in each of the three segments to find the branching point that obtains the least total volume.

From our analysis of Case 1, we know that the minimums of $V_1(c_1)$ and $V_3(c_1)$ occur at $q_1$ and $q_3$ respectively.  We compute that the minimum of $V_4(c_1)$ occurs at the midpoint of the whole interval, i.e., at
$$
c_1=\frac{a_1+b_1}{2} = q_2.
$$

As before, the candidate points for the minimizer are $\frac{b_1a_2}{b_2}$, $\frac{a_1b_2}{a_2}$, $q_1$, $q_2$ and $q_3$.   However, by making the following observations we can further reduce the points we need to examine.

If $q_1$ were to be the global minimizer, then it must fall in the appropriate subdomain, i.e., it must be that $q_1 \leq \frac{a_1b_2}{a_2}$.  However, by Lemma \ref{c2q1}, in Case 2 we always have $q_1 \geq \frac{a_1b_2}{a_2}$. Therefore, we can discard $q_1$ as a candidate point for the minimizer because for it to be the minimizer it would have to be exactly equal to $\frac{a_1b_2}{a_2}$, which is already on the list of candidate points.

If $q_3 \geq \frac{b_1a_2}{b_2}$, then $q_2 \geq q_3 \geq \frac{b_1a_2}{b_2}$, and therefore $q_3$ is the minimizer.  This is because neither $q_1$ nor $q_2$ fall in their key intervals; furthermore, by definition of $q_3$ as the minimizer of $V_3$, we must have that $V_3(q_3) \leq V_3\left(\frac{b_1a_2}{b_2}\right)$, and by Lemma \ref{HlessB2} we know that $V_3\left(\frac{b_1a_2}{b_2}\right) \leq V_1\left(\frac{a_1b_2}{a_2}\right)$.

If this does not occur, i.e. $q_3 < \frac{b_1a_2}{b_2}$, then if $\frac{a_1b_2}{a_2} \leq \frac{a_1+b_1}{2} \leq \frac{b_1a_2}{b_2}$, the midpoint $q_2$ is the minimizer.  This is because under these conditions, $q_2$ is the only minimizer that occurs in the `correct' function piece, and by definition of $q_2$ as the minimizer of $V_4$, the function value is no more than at either of the end points.

Otherwise, we have that $\frac{b_1a_2}{b_2}$ is the minimizer by Lemma \ref{HlessB2}.  Therefore Procedure \ref{x1flowchart} is correct.
\qed
\end{proof}

\subsection{Side note}

For completeness, and as an interesting side point, we  note that in Case 1, if it were possible to have $q_1 \leq \frac{b_1a_2}{b_2}$, then $q_3 \leq q_2 \leq q_1 \leq \frac{b_1a_2}{b_2}$, and therefore $q_1$ would be the minimizer.  This is because neither $q_2$ nor $q_3$ would fall in their key intervals; furthermore, by the definition of $q_1$ as the minimizer of $V_1$, we have that $V_1(q_1) \leq V_1\left(\frac{b_1a_2}{b_2}\right)$, and by Proposition \ref{lemAlessH}(see the Appendix), we know that $V_1(q_1) \leq V_2\left(\frac{a_1b_2}{a_2}\right)$.  However, by Lemma \ref{c1q1} we have already discarded this case.

As another interesting side point, we also note that in Case 2, if it were possible to have $q_1 \leq \frac{a_1b_2}{a_2}$, then $q_3 \leq q_2 \leq q_1 \leq \frac{a_1b_2}{a_2}$, and $q_1$ would be the minimizer.  This is because neither $q_2$ nor $q_3$ would fall in their key intervals.  Furthermore, by definition of $q_1$ as the minimizer of $V_1$, we must have that $V_1(q_1) \leq V_1\left(\frac{a_1b_2}{a_2}\right)$, and by Proposition \ref{lemAlessH2} (see the Appendix), we know that $V_1(q_1) \leq V_4\left(\frac{b_1a_2}{b_2}\right)$.  However, by Lemma \ref{c2q1} we have already discarded this case.

\subsection{Some examples}

We can illustrate these piecewise-quadratic functions for the possible outcomes of Procedure 1. In this illustration, we focus on Case 1, and therefore Figure \ref{fig:eg} shows the function $TV(c_1)$ over the domain $[a_1,b_1]$.  The (orange) dashed curve illustrates an example where the minimizer of $V_3(c_1)$, (i.e. $q_3$), falls in the relevant interval, and therefore is the minimizer over our whole domain.  The (purple) solid curve illustrates an example where $q_3$ does not fall in this interval, however the midpoint, $q_2$, falls in between the quantities $\frac{b_1a_2}{b_2}$ and $\frac{a_1b_2}{a_2}$ and is therefore the required minimizer.  The (green) dotted curve illustrates an example where neither of the above happens, and therefore the breakpoint between the function $V_2(c_1)$ and the function $V_3(c_1)$ is the minimizer.  In this example we are in Case 1, and therefore this point is $\frac{a_1b_2}{a_2}$.

\begin{figure}[h!b]
\begin{center}
\includegraphics[width=0.9\textwidth,keepaspectratio]{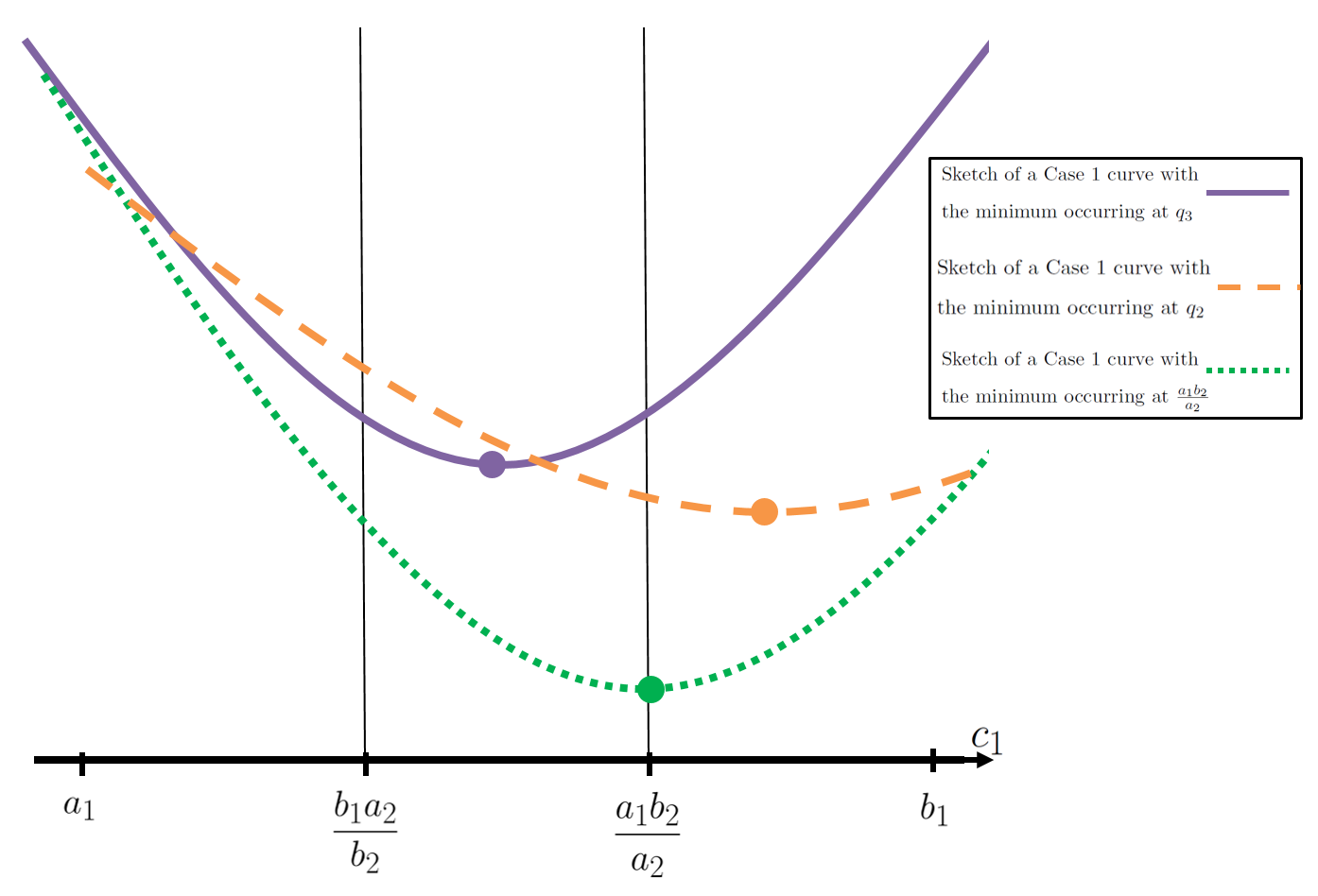}
    \caption{Picture to illustrate the possible outcomes of Procedure 1 in Case 1}
     \label{fig:eg}
     \end{center}
 \end{figure}

% \newpage

It is important to note that each of the cases in Procedure 1 actually \emph{can} occur.  It is easy to check the following:
\begin{itemize}
 \item An example of a dashed curve (minimum occurs at $q_3$) is ($a_1=1$, $b_1=35$, $a_2=2$,  $b_2=12$, $a_3=12$, $b_3=35$).
 \item An example of a solid curve (minimum occurs at $q_2$) is ($a_1=1$, $b_1=34$, $a_2=2$, $b_2=36$, $a_3=12$, $b_3=35$).
 \item An example of a dotted curve (minimum occurs at $\frac{a_1b_2}{a_2}$) is ($a_1=1$, $b_1=8$, $a_2=5$, $b_2=22$, $a_3=1$, $b_3=4$).
\end{itemize}
 Unfortunately, the plots of the actual functions do not display the key details as clearly as our illustration, so we do not include them here.

Furthermore, an example of Case 2, where the minimum occurs at the breakpoint between the function $V_4$ and the function $V_3$, i.e. the point $\frac{b_1a_2}{b_2}$ is ($a_1=1$, $b_1=13$, $a_2=1$, $b_2=2$, $a_3=2$, $b_3=4$).  Finally, a simple example of Case 0, is the special case ($a_1=0$, $b_1=1$, $a_2=0$, $b_2=1$, $a_3=0$, $b_3=1$).  In Figure \ref{fig:01eg} we can see the plot of this function and the minimum, which falls at the midpoint.  In Case 0 we always have $q_1=q_2=q_3=\frac{a_1+b_1}{2}=\frac{b_1}{2}$.

% \vspace{20mm}

 \begin{figure}[h!]
 \begin{center}
  \includegraphics[width=0.7\textwidth,keepaspectratio]{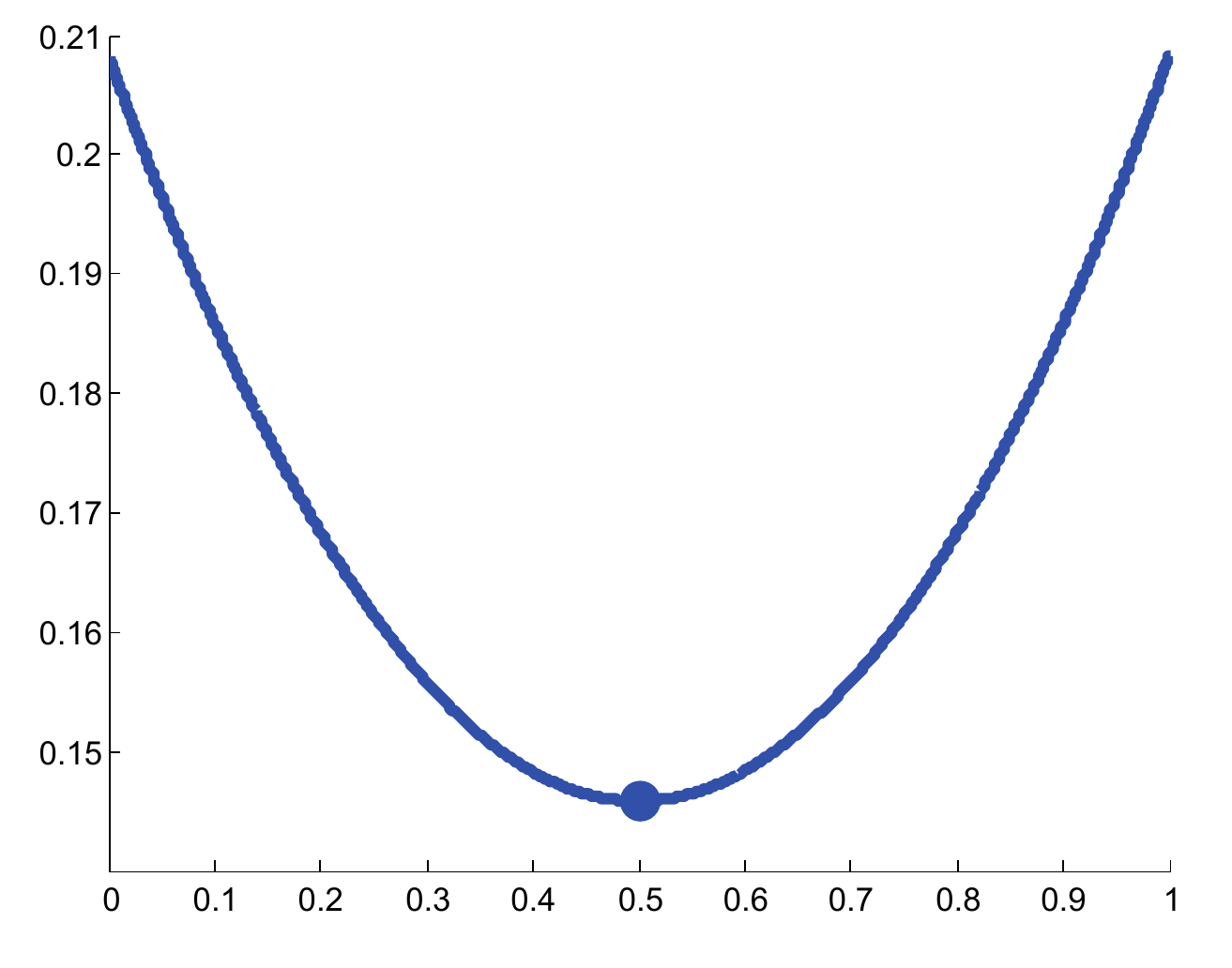}
    \caption{Plot of the total volume function (when branching on $x_1$), for parameter values: $(a_1=0,b_1=1,a_2=0,b_2=1,a_3=0,b_3=1)$}
    \label{fig:01eg}
    \end{center}
 \end{figure}

\subsection{Global convexity of our piecewise-quadratic function over its domain}
\label{sec:globcon}
We have seen that each piece of $TV(c_1)$ and $ \widehat{TV}(c_1)$ is a convex quadratic function.  However, this does not imply that the functions are convex over the whole domain, $[a_1,b_1]$.  Nevertheless, as we show in the following theorem, with a bit more work, we are able to demonstrate that $TV(c_1)$ and $ \widehat{TV}(c_1)$ \emph{are} convex over the domain, $[a_1,b_1]$.  It is very useful that these functions are globally convex; if a variable appears in many trilinear terms,
it is quite reasonable to combine volumes in a reasonable manner
(see \cite{SpeakmanYuLee2016}).
For example, we can take a weighted average (of the sum of the two volumes for each term) as a measure for deciding on a branching point.  A weighted average (assuming positive weights) of convex functions is convex, and therefore, the global-convexity property of these functions allows us to find the optimal branching point (defined as the minimum of the weighted-average function) by a simple bisection search.  However, it is important to note that we are not advocating a bisection search if there is only one term being considered.  In this case, Procedure (\ref{fig:flowchart}) is more efficient.

%\begin{figure}[h!]
%\begin{center}
%    \includegraphics[width=0.6\textwidth,keepaspectratio]{pics/Examplequadraticconvex2}
%    \caption{Illustration of a globally-convex piecewise-quadratic function}
%    \label{fig:piecewisequadconvex}
%     \end{center}
% \end{figure}

\begin{theorem}
\label{globalconvexity}
Given that the upper- and lower-bound parameters respect the labeling \ref{Omega}, the functions $TV(c_1)$ and $ \widehat{TV}(c_1)$ are globally-convex functions in the branching point $c_1$ over the domain $[a_1,b_1]$.
\end{theorem}

\begin{proof}
To demonstrate the global convexity of $TV(c_1)$, we will establish that it is the pointwise maximum of the convex functions $V_1(c_1)$, $V_2(c_1)$ and $V_3(c_1)$.  Similarly, to demonstrate the global convexity of $\widehat{TV}(c_1)$, we will establish that it is the pointwise maximum of the convex functions $V_1(c_1)$, $V_4(c_1)$ and $V_3(c_1)$.

%\newpage

\paragraph{Global convexity of $\bm{TV(c_1)}$:}

Consider the difference of $V_1(c_1)$ and $V_2(c_1)$:

$$V_1(c_1) - V_2(c_1) = \frac{(b_3-a_3)^2(b_1-c_1)(b_2-a_2)(b_1a_2-c_1b_2)}{12}.$$

Note that for all parameter values such that \ref{Omega} is satisfied and $a_1 < c_1 < b_1$, we have that $V_1(c_1) > V_2(c_1)$ if and only if $c_1 < \frac{b_1a_2}{b_2}$ and conversely $V_1(c_1) > V_2(c_1)$ if and only if $c_1 > \frac{b_1a_2}{b_2}$.  They are equal when $c_1 = \frac{b_1a_2}{b_2}$.

Now consider the difference of $V_3(c_1)$ and $V_2(c_1)$:

$$V_3(c_1) - V_2(c_1) = \frac{(b_3-a_3)^2(c_1-a_1)(b_2-a_2)(c_1a_2-a_1b_2)}{12}.$$

Again, note that for all parameter values such that \ref{Omega} is satisfied and $a_1 < c_1 < b_1$, we have that $V_3(c_1) > V_2(c_1)$ if and only if $c_1 > \frac{a_1b_2}{a_2}$ and conversely $V_3(c_1) < V_2(c_1)$ if and only if $c_1 < \frac{a_1b_2}{a_2}$.  They are equal when $c_1 = \frac{a_1b_2}{a_2}$.  Also recall that in the definition of $TV(c_1)$, we implicitly assume $\frac{b_1a_2}{b_2} \geq \frac{a_1b_2}{a_2}$.  We can make the following observations.

On the interval $c_1 \in \left(a_1,\frac{b_1a_2}{b_2}\right)$ we have $V_1(c_1)>V_2(c_1)>V_3(c_1)$, at $c_1=\frac{b_1a_2}{b_2}$ we have $V_1(c_1)=V_2(c_1)>V_3(c_1)$, on the interval $c_1 \in \left(\frac{b_1a_2}{b_2}, \frac{a_1b_2}{a_2}\right)$ we have $V_2(c_1)>V_1(c_1)$ and $V_2(c_1)>V_3(c_1)$. At $c_1=\frac{a_1b_2}{a_2}$ we have $V_3(c_1)=V_2(c_1)>V_1(c_1)$ and on the interval $c_1 \in \left(\frac{a_1b_2}{a_2}, b_1\right)$ we have $V_3(c_1)>V_2(c_1)>V_1(c_1)$.  Furthermore, when $c_1=a_1$, we have $V_1(c_1)>V_2(c_1)=V_3(c_1)$ and when $c_1=b_1$ we have $V_3(c_1)>V_2(c_1)=V_1(c_1)$.

From these observations, it is clear that $TV(c_1)$, is the pointwise maximum of the convex functions $V_1(c_1)$, $V_2(c_1)$ and $V_3(c_1)$ over the domain $[a_1,b_1]$ therefore we observe that $TV(c_1)$ is globally convex over the domain $[a_1,b_1]$.

\paragraph{Global convexity of $\bm{\widehat{TV}(c_1)}$:}

Consider the difference of $V_1(c_1)$ and $V_4(c_1)$:

$$V_1(c_1) - V_4(c_1) = \frac{(b_3-a_3)^2(c_1-a_1)(b_2-a_2)(a_1b_2-c_1a_2)}{12}.$$

Note that for all parameter values such that \ref{Omega} is satisfied and $a_1 < c_1 < b_1$, we have that $V_1(c_1) > V_4(c_1)$ if and only if $c_1 < \frac{a_1b_2}{a_2}$ and conversely $V_1(c_1) > V_2(c_1)$ if and only if $c_1 > \frac{a_1b_2}{a_2}$.  They are equal when $c_1 = \frac{a_1b_2}{a_2}$.

Now consider the difference of $V_3(c_1)$ and $V_4(c_1)$:

$$V_3(c_1) - V_4(c_1) = \frac{(b_3-a_3)^2(b_1-c_1)(b_2-a_2)(c_1b_2-b_1a_2)}{12}.$$

Again, note that for all parameter values such that \ref{Omega} is satisfied and $a_1 < c_1 < b_1$, we have that $V_3(c_1) > V_4(c_1)$ if and only if $c_1 > \frac{b_1a_2}{b_2}$ and conversely $V_3(c_1) < V_2(c_1)$ if and only if $c_1 < \frac{b_1a_2}{b_2}$.  They are equal when $c_1 = \frac{b_1a_2}{b_2}$.

Now recall that $\widehat{TV}(c_1)$ is defined with the assumption $\frac{b_1a_2}{b_2} < \frac{a_1b_2}{a_2}$.  We can make almost identical observations to those above to see that $\widehat{TV}(c_1)$  is the pointwise maximum of the convex functions $V_1(c_1)$, $V_4(c_1)$ and $V_3(c_1)$ and therefore is also globally convex over the domain $[a_1,b_1]$.
\qed \end{proof}

\subsection{Bounds on where the optimal branching point can occur}

We have seen in \S \ref{intro} that software employ methods to avoid selecting a branching point that falls too close to either endpoint of the interval.  Therefore, a natural issue to consider is whether our minimizer can fall close to either of the endpoints.  We want to know how likely it is that solvers are routinely precluding our optimal branching point.  The following theorems give some insight on this issue and show that, in fact, software is unlikely to be cutting off our optimal branching point.

 \begin{theorem}
 \label{upperbound}
 The branching point for variable $x_1$ that obtains the least total volume, never occurs at a point in the interval greater than the midpoint.
 \end{theorem}

\begin{proof}
If $a_2=0$, then we are in Case 0, and the minimizer is at the midpoint, which is clearly no greater than the midpoint.\\

If $\frac{a_1b_2}{a_2}\geq\frac{b_1a_2}{b_2}$, then we are in Case 1.  If $q_3 \geq \frac{a_1b_2}{a_2}$, then $q_3$ is the minimizer, but we know that $q_3 \leq \frac{a_1+b_1}{2}$ (see Equation \ref{q3less}).  If $q_2 = \frac{a_1+b_1}{2}$ falls in the interval $\left[\frac{b_1a_2}{b_2},\frac{a_1b_2}{a_2}\right]$, then the midpoint is the minimizer.  If it does not, then (i) $\frac{a_1b_2}{a_2}$ is the minimizer, and (ii) it must be that either that $\frac{a_1+b_1}{2} > \frac{a_1b_2}{a_2}$, in which case our claim is valid, or  $\frac{a_1+b_1}{2} < \frac{b_1a_2}{b_2} \leq \frac{a_1b_2}{a_2}$.  We will show by contradiction that this cannot be the case.

Toward this end, assume that:
\[
\frac{a_1+b_1}{2} < \frac{b_1a_2}{b_2}\quad \text{and} \quad \frac{a_1+b_1}{2} < \frac{a_1b_2}{a_2}.
\]
This implies:
\begin{align*}&2b_1a_2 -b_1b_2 - a_1b_2 = b_1(a_2-b_2) + (b_1a_2-a_1b_2) > 0,  \quad \text{and}\\
&2a_1b_2-a_1a_2-b_1a_2 = a_1(b_2-a_2) +(a_1b_2-b_1a_2) > 0.
\end{align*}
Now let $X\colonequals b_2-a_2$ and $Y\colonequals b_1a_2-a_1b_2$ (note that both $X$ and $Y$ are non-negative: Lemma \ref{lem91}).  Therefore we can write our assumption as:
  $$b_1(-X) + Y > 0   \quad \text{and} \quad a_1(X) + (-Y) > 0,$$
which implies $$Y > b_1X \quad \text{and} \quad Y < a_1X,$$ a contradiction.  Therefore, in Case 1 the minimizer must be no larger than the midpoint.

We make a similar argument for Case 2.  Here  $\frac{a_1b_2}{a_2}<\frac{b_1a_2}{b_2}$.  If $q_3 \geq \frac{b_1a_2}{b_2}$, then $q_3$ is the minimizer, but we know that $q_3 \leq \frac{a_1+b_1}{2}$ (see Equation \ref{q3less}).  If $q_2 = \frac{a_1+b_1}{2}$ falls in the interval $\left[\frac{a_1b_2}{a_2},\frac{b_1a_2}{b_2}\right]$, then the midpoint is the minimizer.  If it does not, then (i) $\frac{b_1a_2}{b_2}$ is the minimizer, and (ii) it must be that either that $\frac{a_1+b_1}{2} > \frac{b_1a_2}{b_2}$, in which case our claim is valid, or  $\frac{a_1+b_1}{2} < \frac{a_1b_2}{a_2} < \frac{b_1a_2}{b_2}$.  However, we have just shown by contradiction that this cannot be the case.  Therefore, in Case 2 the minimizer must be no larger than the midpoint.
 \qed \end{proof}

This theorem gives an upper bound on the fraction through the interval the minimizer can fall (namely $\frac{1}{2}$).  Furthermore, this bound is sharp (i.e. it is obtained and therefore cannot be strengthened) because we know examples when the minimizer is exactly at the midpoint.  It would be nice to also obtain a sharp lower bound on this fraction.  By demonstrating that the minimizer cannot fall too close to the end points of the interval, we are providing mathematical evidence to justify the current choices of branching point in software, as discussed in \S\ref{intro}.  The following theorem gives a lower bound on this fraction when $a_2 \neq 0$, (when $a_2 =0$, we know that the minimizer will be exactly at the midpoint).  We note that because of the condition \ref{Omega}, the problem is no longer symmetric and therefore knowledge about the upper bound does not allow us to draw conclusions about the lower bound.

 \begin{theorem}
Given upper- and lower-bound parameters ($a_1$, $b_1$, $a_2$, $b_2$, $a_3$, $b_3$) satisfying \ref{Omega}, and $a_2 \neq 0$.  The branching point for variable $x_1$ that obtains the least total volume, never occurs at a point in the interval less than
$$\min\left\{\max\left \{\frac{a_1(b_2-a_2)}{a_2(b_1-a_1)},\frac{b_1a_2-a_1b_2}{b_1b_2-a_1b_2}\right\}, \frac{1}{2}\right\} $$
of the way through the interval.
  \end{theorem}

   \begin{proof}
There are four candidate points where the minimizer can occur.  Namely, $q_2=\frac{a_1+b_1}{2}$, $q_3$, $\frac{a_1b_2}{a_2}$, and $\frac{b_1a_2}{b_2}$.  Therefore

$$\min \left \{\frac{a_1+b_1}{2}, q_3, \frac{a_1b_2}{a_2}, \frac{b_1a_2}{b_2} \right\}, $$
is a trivial lower bound on this minimizer.

We know that if $q_3$ is the minimizer, then we must have $q_3 \geq \frac{a_1b_2}{a_2}$ (Case 1), or $q_3 \geq \frac{b_1a_2}{b_2}$ (Case 2), so we can discard this point.

Additionally, we know that if $\frac{a_1b_2}{a_2}$ is the minimizer, then we have $\frac{a_1b_2}{a_2} \geq \frac{b_1a_2}{b_2}$ (Case 1), and if $\frac{b_1a_2}{b_2}$ is the minimizer, then we have $\frac{b_1a_2}{b_2} > \frac{a_1b_2}{a_2}$ (Case 2).

Therefore we have that a lower bound on the minimizer is:

$$\min\left\{\max\left \{\frac{a_1b_2}{a_2},\frac{b_1a_2}{b_2}\right\}, \frac{a_1+b_1}{2}\right\}. $$
Moreover, a lower bound for the fraction of the interval where this point can fall is:

$$\min\left\{\max\left \{\frac{\frac{a_1b_2}{a_2} - a_1}{b_1-a_1},\frac{\frac{b_1a_2}{b_2}-a_1}{b_1-a_1}\right\}, \frac{\frac{a_1+b_1}{2}-a_1}{b_1-a_1}\right\} $$

$$=\min\left\{\max\left \{\frac{a_1(b_2-a_2)}{a_2(b_1-a_1)},\frac{b_1a_2-a_1b_2}{b_1b_2-a_1b_2}\right\}, \frac{1}{2}\right\}.$$
\qed \end{proof}

We note that this lower bound is unlikely to be sharp.  Consider the case where $a_1=0$, $a_2= \epsilon>0$ and $b_2=1$.  This bound becomes $\epsilon$, and is therefore not particularly informative, given that we can make $\epsilon$ as close to zero as we wish.  However, we have computationally checked many examples, and we have yet to find an example where the minimizer occurs less than $\sim 0.45$ of the way through the interval.  It would be nice to sharpen this bound, and our computations indicate that this should be possible.

\section{Branching on $x_2$ and $x_3$}
\label{x2x3}

We noted in \S\ref{x1} that because of the structure of the volume function of the convex hull, the second and third variables are interchangeable.  Therefore, the branching-point analyses for these variables will be equivalent.  To see how the results in this case are less complex than in the $x_1$ case, recall the condition \ref{Omega}, which due to our non-negativity assumption can be written as  $$\frac{a_1}{b_1} \leq \frac{a_2}{b_2} \leq \frac{a_3}{b_3}. $$

Now consider what happens to the quantity $\frac{a_2}{b_2}$ when we branch on $x_2$.  In the left interval, $a_2$ remains constant, and $b_2$ becomes the branching point, $c_2 < b_2$.  Therefore, $\frac{a_2}{b_2}$ cannot decrease further.  In the right interval $b_2$ remains constant and $a_2$ becomes the branching point, $c_2 > a_2$.  Therefore, again, $\frac{a_2}{b_2}$ cannot decrease further.  Because of this, the labeling for $x_1$ and $x_2$ will not have to be switched to ensure \ref{Omega} remains satisfied.  Furthermore, $x_2$ and $x_3$ are interchangeable in the formula, so we do not need to consider what happens when the ratios change such that $\frac{a_2}{b_2}>\frac{a_3}{b_3}$.

The case of $x_2$ and $x_3$ therefore both require the analysis of only one convex quadratic function.  This is formalized in the following theorem.

\begin{theorem}
\label{Hullmidpoint23}
Let $c_i \in [a_i,b_i]$ be the branching point for $x_i$, $i=2,3$.  With the convex-hull relaxation, the least total volume after branching is obtained when $c_i=(a_i+b_i)/2$, i.e., branching at the midpoint is optimal.
\end{theorem}

\begin{proof}
We first consider branching on $x_2$.  Consider the sum of the two resulting volumes, given by the following function:
\begin{align*}
TV_2(c_2)=V(a_1,b_1,c_2,b_2,a_3,b_3)+V(a_1,b_1,a_2,c_2,a_3,b_3),
\end{align*}
which is quadratic in $c_2$.  The leading coefficient (i.e. second derivative) is $$TV_2(c_2)=\frac{1}{12}(b_1-a_1)(b_3-a_3)(3(b_1b_3-a_1a_3)+(b_1a_3-a_1b_3)),$$ which is greater than or equal to zero for all parameters satisfying \ref{Omega} and hence all $c_2 \in [a_2,b_2]$ (Lemma \ref{lem91}).  Therefore this function is convex.
Setting the first derivative equal to zero and solving for $c_2$, we obtain that the minimum occurs at $c_2=(a_2+b_2)/2$.  Similar analysis can be completed for $i=3$ to obtain the result. \qed
\end{proof}

\section{The optimal branching variable}
\label{compare}

Now that we have established the optimal branching point for each variable in all cases, it is interesting to compare the total volumes obtained when branching at the optimal point for \emph{each} variable.  In this section we establish the optimal branching variable.

\begin{theorem}
\label{Hullvariable1}
Given that the upper- and lower-bound parameters respect the labeling \ref{Omega}, if we assume optimal branching-point selection, then branching on $x_1$ obtains the least total volume, and branching on $x_3$ obtains the greatest total volume. Additionally, even if we branch at the midpoint for $x_1$ (which may not be optimal), this is at least as good as doing optimal branching-point selection (i.e., midpoint branching) on either $x_2$ or $x_3$.
\end{theorem}

\begin{proof}

First, we establish that branching optimally (at the midpoint) on variable $x_2$ obtains a lower total volume than branching optimally (at the midpoint) on variable $x_3$.

The optimal total volume when branching on variable $x_3$ is:
\begin{align*}&\frac{(b_3-a_3)(b_2-a_2)(b_1-a_1)}{48} \;\; \times \\ &\hspace*{0mm} (7a_1a_2a_3+a_1a_2b_3-3a_1a_3b_2-5a_1b_2b_3-5a_2a_3b_1-3a_2b_1b_3+a_3b_1b_2+7b_1b_2b_3).\end{align*}

The optimal total volume when branching on variable $x_2$ is:
\begin{align*}&\frac{(b_3-a_3)(b_2-a_2)(b_1-a_1)}{48} \;\; \times \\ &\hspace*{0mm}(7a_1a_2a_3-3a_1a_2b_3+a_1a_3b_2-5a_1b_2b_3-5a_2a_3b_1+a_2b_1b_3-3a_3b_1b_2+7b_1b_2b_3).\end{align*}

Therefore, the difference in total volume from branching on $x_3$ compared with $x_2$ is: $$\frac{(b_3-a_3)(b_2-a_2)(b_1-a_1)^2(b_2a_3-a_2b_3)}{12},$$ which is greater than or equal to zero by Lemma \ref{lem91}.  Therefore, if we assume optimal branching, branching on $x_3$ always results in a greater volume than branching on $x_2$.

Now let us consider the optimal total volume when branching on $x_1$, this quantity must always be less than or equal to the total volume when branching at the midpoint of the interval (it will be equal exactly when the midpoint is the optimal branching point).  Therefore, if we can establish that branching on variable $x_1$ at the midpoint always obtains a lesser total volume than branching on variable $x_2$ at the midpoint, we will have shown our claim.

Recall Figure \ref{fig:branch}.  We know from the proof of Theorem \ref{upperbound}, that the midpoint can never be less than: $\min\left\{\frac{a_1b_2}{a_2}, \frac{b_1a_2}{b_2}\right\}$.  Therefore, in \emph{every} case, the midpoint must fall in a subdomain where: (i) the labeling for left interval stays the same, and the labeling for the right changes; (ii) the labeling changes for both intervals; or, (iii) the labeling remains the same for both intervals.  This means that we are interested in the function value (total volume) at the midpoint for the functions $V_2(c_1)$, $V_3(c_1)$ and $V_4(c_1)$.

The total volume of branching (on variable $x_1$) at the midpoint if it occurs in the subdomain corresponding to $V_2$ is:
\begin{align*}
&\frac{(b_3-a_3)(b_2-a_2)(b_1-a_1)}{48} \times \\
&(7a_1a_2a_3-3a_1a_2b_3-5a_1a_3b_2+a_1b_2b_3+a_2a_3b_1-5a_2b_1b_3-3a_3b_1b_2+7b_1b_2b_3).
\end{align*}
Therefore, the difference in total volume from branching on $x_2$ compared with this quantity is: $$\frac{(b_3-a_3)^2(b_2-a_2)(b_1-a_1)(b_1a_2-a_1b_2)}{8},$$ which is greater than or equal to zero by Lemma \ref{lem91}.

The total volume of branching (on variable $x_1$) at the midpoint if it occurs in the subdomain corresponding to $V_3$ is:
\begin{align*}
&\frac{(b_3-a_3)(b_2-a_2)(b_1-a_1)}{48} \times \\
&(6a_1a_2a_3-2a_1a_2b_3-3a_1a_3b_2-a_1b_2b_3-4a_2b_1b_3-3a_3b_1b_2+7b_1b_2b_3).
\end{align*}
Therefore, the difference in total volume from branching on $x_2$ compared with this quantity is: $$\frac{(b_3-a_3)^2(b_2-a_2)(b_1-a_1)(4(b_1a_2-a_1b_2)+a_2(b_1-a_1))}{48},$$ which is greater than or equal to zero by Lemma \ref{lem91}.

The total volume of branching (on variable $x_1$) at the midpoint if it occurs in the subdomain corresponding to $V_4$ is:
\begin{align*}
&\frac{(b_3-a_3)(b_2-a_2)(b_1-a_1)}{24} \times \\
&(3a_1a_2a_3-a_1a_2b_3-a_1a_3b_2-a_1b_2b_3-a_2a_3b_1-a_2b_1b_3-a_3b_1b_2+3b_1b_2b_3).
\end{align*}
Therefore, the difference in total volume from branching on $x_2$ compared with this quantity is: $$\frac{(b_3-a_3)^2(b_2-a_2)(b_1-a_1)(b_1b_2-a_1a_2+3(b_1a_2-a_1b_2))}{48},$$ which is greater than or equal to zero by Lemma \ref{lem91}.

Therefore, for each one of these possible scenarios, optimally branching on $x_2$ results in a greater volume than branching on $x_1$ at the midpoint. And so we can conclude that given optimal branching, branching on $x_1$ obtains the least total volume, and branching on $x_3$ obtains the greatest total volume.

 \qed
\end{proof}

\section{Concluding remarks and future work}
\label{conclusion}

We have presented some analytic results on branching variable and branching-point selection in the context of sBB applied to models having functions involving the multiplication of three or more terms.  In particular, for trilinear monomials $f=x_1x_2x_3$ on a box domain satisfying \ref{Omega},
 we have shown that when the convex-hull relaxation is used, and the branching variable is $x_2$ or $x_3$, branching at the commonly-used midpoint results in the least total volume.

We have presented a simple procedure for obtaining the optimal branching point when using the convex-hull relaxation and branching on variable $x_1$.  We have provided a sharp upper bound on where in the interval the minimizer can occur, and we have also obtained a lower bound for this fraction.  By computationally checking many examples, we have evidence to suggest that this lower bound can be sharpened, thus providing analysis that backs up software's current choice of branching point.  Furthermore, we have shown that the piecewise-quadratic functions we have been considering are globally convex over their entire domain.

Given that we branch at an optimal branching point, we have also compared the choice of branching variable.  We demonstrate that branching on $x_1$ gives the least total volume.

We are in the process of  carrying out a similar analysis to what we have done here, but for the best of the double-McCormick convexifications
rather than for the convex-hull relaxation. However, due to the structure of the volume formula for the best double-McCormick convexification (see \cite{SpeakmanLee2015}), our task is significantly more complex.

Finally, we hope that our mathematical results can be used as some guidance
toward justifying, developing and refining practical branching rules.
We believe that our work is just a first step in this direction.
In this regard, we hope to further extend our mathematical analysis
to directly deal with variables appearing in multiple non-linear terms. \\ \\

\begin{acknowledgements}
This work was supported in part by ONR grants N00014-14-1-0315 and N00014-17-1-2296. The authors gratefully acknowledge
conversations with Ruth Misener and Nick Sahinidis concerning
how branching points are selected in  \verb;ANTIGONE; and \verb;BARON;.
\end{acknowledgements}

\bibliographystyle{spmpsci}      % mathematics and physical sciences
\bibliography{bibthesis}   % name your BibTeX data base

\hspace*{55mm}

\appendix{\noindent {\normalsize \bf Appendix: technical propositions and lemmas}\\
\label{app}

In this section, we provide the technical propositions and lemmas used for our analysis.

\begin{proposition}
\label{lemAlessH}
Given that the upper- and lower-bound parameters respect the labeling \ref{Omega}, and $\frac{b_1a_2}{b_2} \leq \frac{a_1b_2}{a_2}$,
\begin{equation*}
V_1(q_1) \leq V_2\left(\frac{a_1b_2}{a_2}\right) = V_3\left(\frac{a_1b_2}{a_2}\right).
\end{equation*}

\end{proposition}

\begin{proof}{}

It is easy to check that $V_2\left(\frac{a_1b_2}{a_2}\right) = V_3\left(\frac{a_1b_2}{a_2}\right)$.

\begin{dmath*}
V_2\left(\frac{a_1b_2}{a_2}\right) - V_1(q_1) = \frac{(b_3-a_3)(b_2-a_2)}{48(4b_2b_3-a_2b_3-3a_2a_3)a_2^2} \times \left (pa_1^2 + qa_1 + r \right),
\end{dmath*}

\noindent where
\begin{align*}
p&=\Big(-3a_2a_3-a_2b_3+b_2a_3+3b_2b_3\Big)\times \\
\qquad \qquad&\Big(-3a_2^3a_3-a_2^3b_3+13a_2^2b_2a_3+7a_2^2b_2b_3-12a_2b_2^2a_3-20a_2b_2^2b_3+16b_2^3b_3\Big)\\
&=\Big(3(b_2b_3-a_2a_3)+b_2a_3-a_2b_3\Big) \times \\
\qquad \qquad &\Big((-3a_2^3+13a_2^2b_2-12a_2b_2^2)a_3 + (-a_2^3+7a_2^2b_2-20a_2b_2^2+16b_2^3)b_3\Big),\\[1.2em]
q&=4a_2b_1(2a_2^2a_3-3a_2b_2a_3-3a_2b_2b_3+4b_2^2b_3) \\ & \qquad\qquad\qquad\qquad\qquad\qquad\qquad\times (3a_2a_3+a_2b_3-b_2a_3-3b_2b_3),\\[1.2em]
r&=4a_2^2b_1^2(a_2a_3+a_2b_3-2b_2b_3)^2.
\end{align*}

To show that $V_2\left(\frac{a_1b_2}{a_2}\right) - V_1(q_1)$ is non-negative for all parameters satisfying \ref{Omega}, we will show that $pa_1^2 + qa_1 + r \geq 0$ for all parameters satisfying \ref{Omega}.

We observe:
\begin{dmath*}
\Big((-a_2^3+7a_2^2b_2-20a_2b_2^2+16b_2^3)b_3 + (-3a_2^3+13a_2^2b_2-12a_2b_2^2)a_3\Big) \equalscolon b_3Y + a_3Z,
\end{dmath*}
where
\begin{equation*}Y + Z = 4(b_2-a_2)(2b_2-a_2)^2 \geq 0, \end{equation*}
and
\begin{equation*}
Y = \bigg(b_2-a_2\bigg)\bigg(4b_2(b_2-a_2)+12b_2^2+a_2^2\bigg) + 2a_2^2b_2 \geq 0.
\end{equation*}

Therefore, by Lemma \ref{lem94} we have that $b_3Y + a_3Z$ is non-negative and so $p$ is non-negative (Lemma \ref{lem91}).  From this we know that $pa_1^2 + qa_1 + r$ is a convex function in $a_1$ and we can find the minimizer by setting the derivative to zero and solving for $a_1$.  The minimum occurs at
$$a_1 = \frac{2b_1a_2(2a_2^2a_3-3a_2b_2a_3-3a_2b_2b_3+4b_2^2b_3)}{(-3a_2^3a_3-a_2^3b_3+13a_2^2b_2a_3+7a_2^2b_2b_3-12a_2b_2^2a_3-20a_2b_2^2b_3+16b_2^3b_3)}.$$

Substituting this in to $pa_1^2 + qa_1 + r$, we obtain that the minimum value of this quadratic is:
$$\frac{4a_2^2b_1^2(b_3-a_3)(b_2-a_2)^3(3a_2a_3+a_2b_3-4b_2b_3)^2}{(-3a_2^3a_3-a_2^3b_3+13a_2^2b_2a_3+7a_2^2b_2b_3-12a_2b_2^2a_3-20a_2b_2^2b_3+16b_2^3b_3)}.$$

In demonstrating the non-negativity of $p$, we have already shown that the denominator is non-negative, and it is easy to see that the numerator is non-negative for all values of the parameters satisfying \ref{Omega}.  Therefore $pa_1^2 + qa_1 + r \geq 0$, and consequently, $V_2\left(\frac{a_1b_2}{a_2}\right) - V_1(q_1) \geq 0$ as required.

\qed \end{proof}

\setcounter{lemma}{0}

\begin{lemma}
%\label{HlessB}
Given that the upper- and lower-bound parameters respect the labeling \ref{Omega}, and $\frac{b_1a_2}{b_2} \leq \frac{a_1b_2}{a_2}$,
\begin{equation*}V_1\left(\frac{b_1a_2}{b_2}\right)=V_2\left(\frac{b_1a_2}{b_2}\right) \geq V_2\left(\frac{a_1b_2}{a_2}\right) = V_3\left(\frac{a_1b_2}{a_2}\right)
\end{equation*}
\end{lemma}

\begin{proof}{}
It is easy to check that $V_1\left(\frac{b_1a_2}{b_2}\right)=V_2\left(\frac{b_1a_2}{b_2}\right)$ and $V_2\left(\frac{a_1b_2}{a_2}\right) = V_3\left(\frac{a_1b_2}{a_2}\right)$.

Furthermore,
\begin{dmath*}
V_2\left(\frac{b_1a_2}{b_2}\right) - V_2\left(\frac{a_1b_2}{a_2}\right) = \frac{(b_3-a_3)(b_2-a_2)^2(b_1a_2-a_1b_2)(a_1b_2^2-a_2^2b_1)(3(b_2b_3-a_2a_3)+b_2a_3-a_2b_3)}{12a_2^2b_2^2} \geq 0,
\end{dmath*}
as required.

\qed \end{proof}

%%%%%%%%%%%%%%%%%%%%%%%%%%%%%%%%%%%%%%%%%%%%%%%%%%%%%%%%%%%%%%%%%%%%%%%%%%%%%%

\begin{proposition}
\label{lemAlessH2}
Given that the upper- and lower-bound parameters respect the labeling \ref{Omega}, and $\frac{b_1a_2}{b_2} > \frac{a_1b_2}{a_2}$,
\begin{equation*}
V_1(q_1) \leq V_4\left(\frac{b_1a_2}{b_2}\right) = V_3\left(\frac{b_1a_2}{b_2}\right).
\end{equation*}

\end{proposition}

\begin{proof}{}

It is easy to check that $V_4\left(\frac{b_1a_2}{b_2}\right) = V_3\left(\frac{b_1a_2}{b_2}\right)$.

\begin{dmath*}
V_4\left(\frac{b_1a_2}{b_2}\right) - V_1(q_1) = \frac{(b_3-a_3)(b_2-a_2)}{48(4b_2b_3-a_2b_3-3a_2a_3)b_2^2} \times \left (pa_1^2 + qa_1 + r \right),
\end{dmath*}

\noindent where
\begin{align*}
p&=b_2^2(5b_2b_3-b_2a_3-a_2b_3-3a_2a_3)^2,\\[1.2em]
q&=8b_1b_2(6a_2^2a_3+2a_2^2b_3-3a_2b_2a_3-9a_2b_2b_3+b_2^2a_3+3b_2^2b_3)(b_2b_3-a_2a_3),\\[1.2em]
r&=16b_1^2(-3a_2^3a_3-a_2^3b_3+3a_2^2b_2a_3+5a_2^2b_2b_3-a_2b_2^2a_3-4a_2b_2^2b_3+b_2^3b_3)\\& \qquad\qquad\qquad\qquad\qquad\qquad\qquad\qquad\qquad\qquad\qquad\qquad\qquad\times(b_2b_3-a_2a_3).
\end{align*}

To show this is non-negative for all parameters satisfying \ref{Omega}, we will show $pa_1^2 + qa_1 + r \geq 0$ for all parameters satisfying \ref{Omega}.

Firstly, we observe that
$$p=b_2^2(5b_2b_3-b_2a_3-a_2b_3-3a_2a_3)^2 \geq 0.$$
From this we know that $pa_1^2 + qa_1 + r$ is a convex function in $a_1$, and we can find the minimizer by setting the derivative to zero and solving for $a_1$.  The minimum occurs at
$$a_1 = \frac{4b_1(6a_2^2a_3+2a_2^2b_3-3a_2b_2a_3-9a_2b_2b_3+b_2^2a_3+3b_2^2b_3)(a_2a_3-b_2b_3)}{b_2(3a_2a_3+a_2b_3+b_2a_3-5b_2b_3)^2}.$$

Substituting this in to $pa_1^2 + qa_1 + r$, we obtain that the minimum value of this quadratic is:
$$\frac{16b_1^2(b_3-a_3)(b_2-a_2)^3(b_2b_3-a_2a_3)(3a_2a_3+a_2b_3-4b_2b_3)^2}{(3a_2a_3+a_2b_3+b_2a_3-5b_2b_3)^2},$$

\noindent which is non-negative for all parameters satisfying \ref{Omega}.  Therefore $pa_1^2 + qa_1 + r \geq 0$, and consequently, $V_4\left(\frac{b_1a_2}{b_2}\right) - V_1(q_1) \geq 0$, as required.

\qed \end{proof}

\begin{lemma}
%\label{HlessB2}
Given that the upper- and lower-bound parameters respect the labeling \ref{Omega}, and $\frac{b_1a_2}{b_2} > \frac{a_1b_2}{a_2}$,
\begin{equation*}
V_1\left(\frac{a_1b_2}{a_2}\right)=V_4\left(\frac{a_1b_2}{a_2}\right) \geq V_4\left(\frac{b_1a_2}{b_2}\right) = V_3\left(\frac{b_1a_2}{b_2}\right).
\end{equation*}
\end{lemma}

\begin{proof}{}
It is easy to check that $V_1\left(\frac{a_1b_2}{a_2}\right)=V_4\left(\frac{a_1b_2}{a_2}\right)$ and $V_4\left(\frac{b_1a_2}{b_2}\right) = V_3\left(\frac{b_1a_2}{b_2}\right)$.

Furthermore,
\begin{align*}
&V_4\left(\frac{a_1b_2}{a_2}\right)-V_4\left(\frac{b_1a_2}{b_2}\right)\\ &\qquad\qquad\;\; = \frac{(b_3-a_3)(b_2-a_2)^2(b_1a_2^2-a_1b_2^2)(b_1a_2-a_1b_2)(b_2b_3-a_2a_3)}{3a_2^2b_2^2} \geq 0,
\end{align*}
as required.

\qed \end{proof}

 \begin{lemma}
 %\label{c1q1}
 Given that the parameters satisfy the conditions \ref{Omega}, and furthermore, $\frac{b_1a_2}{b_2} \leq \frac{a_1b_2}{a_2}$, we have $$q_1 \geq \frac{b_1a_2}{b_2}.$$
 \end{lemma}

 \begin{proof}

 From the proof of Theorem \ref{upperbound}, we know that the midpoint, $q_2$, cannot be less than \emph{both} $\frac{a_1b_2}{b_1}$ and $\frac{b_1a_2}{b_2}$.  Therefore we have: $$q_2 \geq \min\left \{\frac{a_1b_2}{b_1},\frac{b_1a_2}{b_2}\right \},$$ and because we saw in \ref{minordering} that $q_1 \geq q_2$ we also have $$q_1 \geq \min\left \{\frac{a_1b_2}{b_1},\frac{b_1a_2}{b_2}\right\}.$$  Therefore, under the conditions of the lemma, $q_1 \geq \frac{b_1a_2}{b_2}$ as required.

 \qed \end{proof}

  \begin{lemma}
%  \label{c2q1}
  Given that the parameters satisfy the conditions \ref{Omega}, and furthermore, $\frac{b_1a_2}{b_2} \geq \frac{a_1b_2}{a_2}$, we have $$q_1 \geq \frac{a_1b_2}{a_2}.$$
 \end{lemma}

\begin{proof}

We saw in the proof of Lemma \ref{c1q1} that $$q_1 \geq \min\left \{\frac{a_1b_2}{b_1},\frac{b_1a_2}{b_2}\right\}.$$  Therefore, under the conditions of the lemma, $q_1 \geq \frac{a_1b_2}{a_2}$ as required.

 \qed \end{proof}

For completeness, we state and give proofs of two very simple lemmas (from \cite{SpeakmanLee2015}) which we used several times.

\begin{lemma}[Lemma 10.1 in \cite{SpeakmanLee2015}]
\label{lem91}
For all choices of  parameters $0\leq a_i < b_i$ satisfying \ref{Omega}, we have: $b_1a_2-a_1b_2\geq 0$, $b_1a_3-a_1b_3 \geq 0$ and $b_2a_3-a_2b_3 \geq 0$.
\end{lemma}

\begin{proof}
$(b_3-a_3)(b_1a_2-a_1b_2) = b_1a_2b_3 + a_1b_2a_3 - a_1b_2b_3 -b_1a_2a_3 \geq 0$
 by \ref{Omega}.
This implies $b_1a_2-a_1b_2 \geq 0$, because $b_3-a_3 > 0$.
$b_1a_3-a_1b_3 \geq 0$ and $b_2a_3-a_2b_3 \geq 0$ follow from \ref{Omega} in a similar way.
\qed
\end{proof}

\begin{lemma}[Lemma 10.4 in \cite{SpeakmanLee2015}]
\label{lem94}
Let $A,B,C,D \in \R$ with $A \geq B\geq 0$, $C+D \geq 0$, $C \geq 0$.  Then $AC+BD \geq0$.
\end{lemma}

\begin{proof}
$AC+BD \geq B(C+D) \geq 0$.
\qed
\end{proof}}

%% Non-BibTeX users please use
%\begin{thebibliography}{}
%%
%% and use \bibitem to create references. Consult the Instructions
%% for authors for reference list style.
%\end{thebibliography}

\end{document}